\documentclass[reqno,oneside]{amsart}
\usepackage{geometry}
\usepackage{csquotes}
\usepackage{graphicx} 
\usepackage{todonotes}
\usepackage{mathtools}
\usepackage{esint}
\usepackage[bookmarks = true,colorlinks = true]{hyperref}
\usepackage{esint}

\theoremstyle{plain}
\newtheorem{proposition}{Proposition}[section]
\newtheorem{theorem}{Theorem}[section]
\newtheorem{lemma}{Lemma}[section]
\newtheorem{remark}{Remark}[section]

\newtheorem{definition}{Definition}[section]

\newcommand{\R}{\mathbb{R}}
\newcommand{\N}{\mathbb{N}}
\newcommand{\sam}{\mathbf{S}}

\allowdisplaybreaks

\title[The Follow-The-Leader scheme with non-monotone velocity]{The Follow-The-Leader scheme with non-monotone velocity}
\author{Marco Di Francesco}
\address{Dipartimento di Ingegneria e Scienze dell'Informazione e Matematica (DISIM), Università degli Studi dell'Aquila, Via Vetoio, Coppito, 67100 L'Aquila, Italy}
\email{marco.difrancesco@univaq.it}
\subjclass[2020]{35L65, 35Q70, 65M12, 90B20}

\keywords{Scalar conservation laws, follow-the-leader models, Hamilton--Jacobi equations, deterministic particle approximation, non-monotone velocity}

\date{July 2026}

\begin{document}

\begin{abstract}
   We study the deterministic particle approximation of entropy solutions to the scalar conservation law $\partial_t\rho+\partial_x(\rho v(\rho))=0$ in the case where the velocity $v$ is locally Lipschitz but not necessarily monotone. To the best of our knowledge, all existing rigorous results on the follow-the-leader scheme require $v$ to be monotone, which allows each particle to look either forward or backward at its neighbor. We introduce a new scheme in which the velocity of each particle is the maximum or minimum of $v$ over the interval between the densities of the two adjacent cells, in the spirit of Godunov-type numerical Hamiltonians for Hamilton–Jacobi equations. The scheme reduces to the classical follow-the-leader model when $v$ is monotone. For initial data of bounded variation, we prove a discrete maximum principle, a uniform BV bound and time equicontinuity in $L^1$, and we show that the piecewise constant particle density converges strongly to the unique Kru\v{z}kov entropy solution.
\end{abstract}

\maketitle

\tableofcontents

\section{Introduction}

Nonlinear conservation laws of the form
\begin{equation}\label{eq:intro_CL}
    \partial_t \rho + \partial_x (\rho v(\rho))=0
\end{equation}
are omnipresent in physics, engineering, social sciences, and biology. We refer to \cite{dafermos_book} for a comprehensive introduction to the subject, and to the more recent \cite{rosini_book} for the applications to traffic and pedestrian movements. 

In most of these applications, the velocity field $v=v(\rho)$ is assumed to be a monotone function of the density $\rho$. This is the case, in particular, in traffic and pedestrian flow modeling, where $v$ is typically taken to be non-increasing, reflecting the empirical fact that higher congestion results in lower speed, as well known since \cite{Greenshields1935}. However, this monotonicity assumption is not always justified. Empirical studies of highway traffic reveal, for instance, a \enquote{synchronized flow} phase \cite{Kerner1998} in which the speed-density relation departs from monotone behavior. Similarly, in pedestrian dynamics \cite{Helbing1995}, self-organization phenomena such as lane formation in counterflows may locally increase the average walking speed at moderate densities compared to lower, less organized ones, before congestion dominates at higher densities. 

Conservation laws are known to produce solutions with discontinuities (shocks).
It is well-known since \cite{kruzkov} that in order to have a well-posed Cauchy problem for \eqref{eq:intro_CL} one needs to define \emph{entropy solutions}, that is, very roughly speaking, weak solutions that dissipate any convex functional around shocks. 

One of the many ways to construct entropy solutions to the Cauchy problem for \eqref{eq:intro_CL} is
via \emph{deterministic Lagrangian particles}. This method was rigorously applied for the first time in \cite{DR} in case the map $v$ is \emph{strictly monotone}. For example, in case $v$ is decreasing, the result in \cite{DR} states that, for a given $L^1(\R)\cap L^\infty(\R)$ nonnegative initial datum $\rho_0$ (with possible additional assumptions), given $\left\{\rho^N\right\}_{N\in \N}$ a suitable discrete density reconstruction of an $N$-particle system solving the ODE system
\begin{equation}\label{eq:intro_FTL_1}
    \dot{x}_i(t)=\begin{cases}
        v(R_{i+1}(t)) & \hbox{if $i\in\{0,\ldots,N-1\}$}\\
        v(0) & \hbox{if $i=N$}
    \end{cases}\qquad \hbox{with}\  R_i(t)=\frac{1}{N(x_i(t)-x_{i-1}(t))}
\end{equation}
with initial data obtained by suitably sampling the initial condition $\rho_0$, then $\rho^N$
converges strongly in $L^1_{\mathrm{loc}}([0,+\infty)\times \R)$ towards the unique entropy solution to \eqref{eq:intro_CL} with initial datum $\rho_0$ as $N\rightarrow+\infty$. The choice of the forward density $R_{i+1}$ in the ODE \eqref{eq:intro_FTL_1} has a quite natural interpretation in traffic flow, where it is the case that $v$ is decreasing. A vehicle will adjust its speed depending on the distance from the preceding vehicle, and completely ignore the following one. In fact, there is also a mathematical justification behind this choice, as we will explain later in this introductory chapter. 

The result in \cite{DR} holds for compactly supported initial data and assuming further either $\rho_0\in BV(\R)$ or that $\rho \mapsto \rho v'(\rho)$ is non-increasing. The assumption of compactly supported initial data was replaced by the assumption of $\rho_0$ having finite first moment in \cite{DFR}. The case with Dirichlet boundary conditions on an interval was considered in \cite{DFRR2}. The main result for the Cauchy problem was later proven by following an alternative strategy in \cite{HR1,HR2} in the case of initial data that are far from vacuum. The case of solutions with infinite mass, leading to infinitely many particles, was addressed in \cite{Marc}. The time-discrete version of \eqref{eq:intro_FTL_1} was considered in \cite{HR2} and later in \cite{DFIR}. This approach has been extended also to the Hughes model for pedestrian movements \cite{DFRR,ARS}, to the case of nonlocal scalar conservation laws for interacting particle systems \cite{DFRad,FR,RS}, and to the case of models on networks \cite{Card}. In the simplest case of a scalar conservation law, the many particle system \eqref{eq:intro_FTL_1} is often referred to as the \emph{follow-the-leader} model. We also recall similarities with nearest-neighbor deterministic schemes for linear and nonlinear diffusion models, see \cite{russo,gosse,matthes,DM}.

While the monotonicity of $v$ plays no significant role in the existence,  uniqueness and regularity theory for \eqref{eq:intro_CL} - a crucial property for that being, on the other hand, the convexity of the flux $f(\rho)=\rho v(\rho)$ - all the above mentioned results on deterministic particle approximations rely quite heavily on the assumption of $v$ being \emph{monotone}. A problem that has remained open for some time is whether or not such monotonicity assumption on $v$ may be removed, at least in the case of the plain conservation law \eqref{eq:intro_CL}. 

In this paper, we provide a positive answer to said problem. Clearly, the very same follow-the-leader model \eqref{eq:intro_FTL_1} cannot work if $v$ is non-monotone. It has been already observed in \cite{DR} that if $v$ is increasing then the scheme \eqref{eq:intro_FTL_1} should be replaced by $\dot{x}_i(t)=v(R_i(t))$, that is, every particle $x_i$ is affected by the distance from the particle $x_{i-1}$. If the monotonicity of $v$ changes, one should switch from backward to forward values of the discrete densities $R_i(t)$ and vice-versa in the scheme depending on the values of those densities, which creates problems even to prove the well-posedness of the discrete scheme. 

A new scheme has to be designed in case $v$ is non monotone. In this paper we propose the following one:
\begin{align}\label{eq:intro_FTL_2}
& \dot{x}_i(t)=
\begin{cases}
    \displaystyle{\min_{R\in[0,R_1(t)]}v(R)} & \hbox{if $i=0$}\\
    \begin{cases}
\displaystyle{\max_{R\in[R_{i+1}(t),R_i(t)]}v(R)} & \quad\hbox{if $R_i(t)\geq R_{i+1}(t)$}\\
        \displaystyle{\min_{R\in [R_i(t),R_{i+1}(t)]}v(R)} &\quad \hbox{if $R_i(t)<R_{i+1}(t)$}
    \end{cases} & \hbox{if $i\in \{1,\ldots,N-1\}$}\\
    \displaystyle{\max_{R\in [0,R_N(t)]}v(R)} & \hbox{if $i=N$}\,.
\end{cases}
\end{align}
The scheme \eqref{eq:intro_FTL_2} reduces to the standard scheme \eqref{eq:intro_FTL_1} in case $v$ is  decreasing (and to its "companion", backward looking model in case $v$ is increasing).

The genesis of scheme \eqref{eq:intro_FTL_2} comes from the analysis of the existing literature on rather well known techniques to approximate Hamilton-Jacobi type equations via finite differences, see \cite{bardi_osher,crandall_lions}. We next outline the link between a Follow-the-Leader type scheme and Hamilton-Jacobi type equations. As it was already noticed in \cite{DR}, for a given solution $\rho\geq 0$ to \eqref{eq:intro_CL} (which we assume to have unit mass for simplicity), the pseudo-inverse $X:[0,+\infty)\times [0,1]\rightarrow \R$ of the cumulative distribution function $F(x,t)=\int_{-\infty}^x \rho(t,y) dy$ with respect to the $x$-variable formally satisfies the PDE
\begin{equation}\label{eq:pseudo}
    \partial_t X(t,z)=v\left(\frac{1}{\partial_z X(t,z)}\right)\,,
\end{equation}
which is of Hamilton-Jacobi type. We stress that the variable $z\in [0,1]$ here represents a \emph{cumulative mass}, not a position in space. Assuming for simplicity that $v$ is decreasing, having solved \eqref{eq:intro_FTL_1} globally in time and having set
\[X^N(t,z)=\sum_{i=1}^N\left(x_{i-1}(t)+R_i(t)^{-1}\left(z-\frac{i-1}{N}\right)\right)\mathbf{1}_{[(i-1)/N,i/N)}(z)\,,\]
we observe that $X^N$ satisfies, for $z\in [(i-1)/N,i/N)$,
\begin{align*}
    & \partial_t X^N(t,z)=\left(1-(Nz-(i-1))\right)\dot{x}_{i-1}(t)+\left(Nz(i-1)\right)\dot{x}_i(t)=\\
    & \ = \left(1-(Nz-(i-1))\right)v\left((D^+ X^N(t,z))^{-1}\right) + \left(Nz(i-1)\right)v\left((D^+ X^N(t,i/N))^{-1}\right)\,.
\end{align*}
Note that the above right-hand side is a convex combination of $v(D^+ X(t,z))$ computed at $z$ and at its closest-to-the-right \enquote{grid point} $i/N$. Hence, the proposed particle approximation problem may (almost) be interpreted as a finite difference approximation (on a stationary grid) for the Hamilton-Jacobi type PDE \eqref{eq:pseudo}. 

For models such as \eqref{eq:pseudo}, the use of $\min$ and $\max$ in the definition of a suitable \enquote{numerical Hamiltonian} to design a proper finite difference scheme is quite standard. Such a remark naturally leads to formulate the scheme \eqref{eq:intro_FTL_2}. The scheme proposed here is also reminiscent of the Godunov scheme for scalar conservation laws, see \cite{leveque,HR_book}. Indeed, the aforementioned schemes for Hamilton-Jacobi equations can be formulated as a general Godunov-type scheme (for fluxes subject to no monotonicity or convexity assumptions) for the scalar conservation law obtained by differentiating \eqref{eq:pseudo} with respect to $z$, i.e.
\begin{equation}\label{eq:pseudo1}
    \partial_t (\partial_z X)=\partial_z \left(v(1/\partial_z X)\right)\,.
\end{equation}
The equation \eqref{eq:pseudo1} also leads to a deeper understanding of the role of the monotonicity of $v$ in the design of Follow-the-Leader schemes: setting $R(t,z)=\left(\partial_z X(t,z)\right)^{-1}$, \eqref{eq:pseudo1} formally leads to the conservation law in non conservative form
\[\partial_t R + R^2 v'(R)\partial_z R = 0\,,\]
and the upwind direction is determined by the sign of $v'$.

In the present paper we essentially follow the program in \cite{DR} (or, more precisely, its shorter version in \cite{DFR}) to prove the necessary estimates on the approximated density and to show that the scheme \eqref{eq:intro_FTL_2} suitably approximates entropy solutions to \eqref{eq:intro_CL}. In Section \ref{sec:scheme} we prove the well-posedness of the scheme \eqref{eq:intro_FTL_2}, by showing in particular a discrete maximum principle in the spirit of \cite[Lemma 1]{DR}. We then prove the scheme \eqref{eq:intro_FTL_2} is contractive in $BV$ for the approximated density
\[\rho^N(t,x)=\sum_{i=1}^N R_i(t)\mathbf{1}_{[x_{i-1}(t),x_i(t))}(x)\]
and satisfies a time equi-continuity property in $L^1$. These properties imply the strong $L^1$-compactness of the scheme. The full atomization scheme is presented in Section \ref{sec:entropy}, which also contains the main convergence result, i.e. Theorem \ref{thm:main}.

\section{The Follow-the-Leader model}\label{sec:scheme}
This section introduces the new version of Follow-the-Leader model proposed in this paper. The approximation of the conservation law \eqref{eq:intro_CL} is addressed in a later section. Here and throughout the whole paper, we shall prescribe the following standing assumption on the velocity map:
\begin{itemize}
    \item [(Vel)] $v:[0,+\infty)\rightarrow \R$ is locally Lipschitz continuous.
\end{itemize}
For a given $N\in \N$, let us define
\[\mathcal{K}_N=\left\{X=(x_0,\ldots,x_N)\in \R^{N+1}\,:\,\, \hbox{$x_i\leq x_{i+1}$ for all $i\in\{0,\ldots,N-1\}$}\right\}\,.
\]
We observe that $\mathcal{K}_N$ is closed in $\R^{N+1}$. We define the map $\mathcal{R}:\mathring{\mathcal{K}}_N\rightarrow [0,+\infty)^N$ by denoting, for all $X=(x_0,\ldots,x_N)\in \mathcal{K}_N$,
\begin{equation}\label{eq:R1}
    \mathcal{R}[X]=(R_1[X],\ldots,R_N[X])
\end{equation}
and by defining
\begin{equation}\label{eq:R2}
    R_i[X]=\frac{1}{N(x_i-x_{i-1})}\qquad \hbox{for all $i\in\{1,\ldots,N\}$}\,.
\end{equation}
Let $v:\R\rightarrow \R$ satisfy (Vel) above. For a given vector $\mathrm{R}=(R_1,\ldots,R_N)\in [0,+\infty)^N$, we define the velocity field $V[\mathrm{R}]\in \R^{N+1}$ as
\[V[\mathrm{R}]=(V_0[\mathrm{R}],\ldots,V_N[\mathrm{R}])\]
with
\begin{align*}
    & V_0[\mathrm{R}]=\min_{R\in [0,R_1]}v(R)\\
    & V_i[\mathrm{R}]=\begin{cases}
\displaystyle{\max_{R\in[R_{i+1},R_i]}v(R)} & \qquad\hbox{if $R_i\geq R_{i+1}$}\\
        \displaystyle{\min_{R\in [R_i,R_{i+1}]}v(R)} &\qquad \hbox{if $R_i<R_{i+1}$}
    \end{cases} \qquad\hbox{for $i\in \{1,\ldots,N-1\}$}\\
    & V_N[\mathrm{R}]=\max_{R\in [0,R_N]}v(R)
\end{align*}

\begin{lemma}\label{lem:max_min}
    Assume $v$ satisfies (Vel). Let $F:[0,+\infty)^2\rightarrow \R$ be defined by 
    \[F((a,b))=\begin{cases}
        \max_{c\in [b,a]}v(c) & \hbox{if $b\leq a$}\\
        \min_{c\in [a,b]}v(c) & \hbox{if $a<b$}
    \end{cases}\]
    Then, $F$ is a locally Lipschitz function.
\end{lemma}
The proof of Lemma \ref{lem:max_min} is provided in the Appendix. We observe that
\begin{align*}
    & V_0[\mathrm{R}]=F(0,R_1)\\
    & V_i[\mathrm{R}]=F(R_i,R_{i+1})\qquad \hbox{for $i\in\{1,\ldots,N-1\}$}\\
    & V_N[\mathrm{R}]=F(0,R_N)\,.
\end{align*}
Hence, Lemma \ref{lem:max_min} and the assumption (Vel) imply that the map $V:[0,+\infty)^N\rightarrow \R^{N+1}$ is locally Lipschitz. Therefore, the map $V\circ \mathcal{R}:\mathring{\mathcal{K}}_N\rightarrow \R^{N+1}$ is locally Lipschitz continuous. Hence, if we fix $X^0=(x_0^0,\ldots,x_N^0)\in \mathring{\mathcal{K}}_N$, the Cauchy problem
\begin{equation}\label{eq:FTL_cauchy}
\begin{cases}
     \dot{X}(t)=V[\mathcal{R}[X(t)]] & \\
 X(0)=X^0
\end{cases}
\end{equation}
admits a unique local-in-time solution by standard Cauchy-Lipschitz theory. For the reader's convenience, we observe that the Cauchy problem \eqref{eq:FTL_cauchy} reads, in its $N+1$ components, as follows
\begin{align}
    & \dot{x}_0(t)=\min_{R\in[0,R_1(t)]}v(R) \nonumber\\
    & \dot{x}_i(t)=
    \begin{cases}
\displaystyle{\max_{R\in[R_{i+1}(t),R_i(t)]}v(R)} & \qquad\hbox{if $R_i(t)\geq R_{i+1}(t)$}\\
        \displaystyle{\min_{R\in [R_i(t),R_{i+1}(t)]}v(R)} &\qquad \hbox{if $R_i(t)<R_{i+1}(t)$}
    \end{cases} \qquad \hbox{for $i\in \{1,\ldots,N-1\}$}\label{eq:FTL_components}\\
   & \dot{x}_N(t)=\max_{R\in [0,R_N(t)]}v(R)\,.\nonumber
\end{align}
Here we adopt the notation
\begin{equation}\label{eq:short_R}
R_i(t)=R_i[X(t)]=\frac{1}{N(x_i(t)-x_{i-1}(t))}\qquad \hbox{for $i=1,\ldots,N$}\,,
\end{equation}
which slightly differs from the one introduced in \cite{DR}. We believe this new notation is more natural than the one used in \cite{DR}: since there are exactly $N$ discrete densities, it is natural to denote them by $R_1,\ldots,R_N$ (differently from \cite{DR}, in which they were denoted with indexes running from $0$ to $N-1$).
\begin{remark}
\emph{
    In case $v$ is monotone, the expressions on the right hand side of \eqref{eq:FTL_components} simplify significantly. For example, in case $v$ is non-increasing, for $R_i\geq R_{i+1}$ the maximum of $v$ on $[R_{i+1},R_i]$ is achieved on $R_{i+1}$; for $R_i<R_{i+1}$, the minimum of $v$ on $[R_i,R_{i+1}]$ is still achieved on $R_{i+1}$. Hence, we obtain the \enquote{standard} Follow-the-Leader ODE
    \[\dot{x}_i(t)=v(R_{i+1}(t))\qquad \hbox{ for $i=0,\ldots,N$}\]
    adopted in \cite{DR}, with the convention $R_{N+1}(t)=0$. 
    Similarly, one can prove that in case $v$ is non-decreasing, \eqref{eq:FTL_components} becomes
    \[\dot{x}_i(t)=v(R_i(t))\qquad \qquad \hbox{ for $i=0,\ldots,N$}\]
    with the convention $R_0(t)=0$.
}
\end{remark}

We now prove a discrete maximum principle for this scheme, which implies, as a byproduct, the global-in-time existence of the unique solution to \eqref{eq:FTL_cauchy}.

\begin{theorem}[Discrete Maximum Principle]\label{thm:MP}
    Let 
    \begin{equation}\label{eq:R_bar}
        \overline{R}=\max_{i\in\{1,\ldots,N\}}R_i(0)= \frac{1}{N\min_{i\in\{1,\ldots,N\}}(x_i^0-x_{i-1}^0)}\,.\
    \end{equation}
    Then, the solution $X(t)=(x_0(t),\ldots,x_N(t))$ to the Cauchy problem \eqref{eq:FTL_cauchy} exists globally in time and satisfies
    \begin{equation}\label{eq:MP}
        R_i(t)\leq \overline{R}\qquad \hbox{for all $i\in\{1,\ldots,N\}$ and for all $t\geq 0$}.
    \end{equation}
\end{theorem}

\proof
We argue by contradiction and assume that the condition in \eqref{eq:MP} gets violated. Assume, therefore, that there exists a time $t^*\geq 0$ within the local existence lifetime of the solution $X(t)$ such that, for some $i\in \{1,\ldots,N\}$, $R_i(t^*)=\overline{R}$, and for some $\varepsilon>0$ one has $R_i(t)>\overline{R}$ for all $t\in (t^*,t^*+\varepsilon)$. Clearly, there might be more than just index $i$ with this property. To cover the most possible general case, we must assume there exist $h,i\in \{1,\ldots,N\}$, with $h\leq i$, such that $R_j(t^*)=\overline{R}$ and $R(t)>\overline{R}$ for $t\in(t^*,t^*+\varepsilon_j)$ for some $\varepsilon_j>0$, for all $j\in \{h,\ldots,i\}$. Without restriction, we may assume $\varepsilon_j=\varepsilon>0$ for all $j\in \{h,\ldots,i\}$. We then assume that $R_{i+1}(t)\leq \overline{R}$ and $R_{h-1}(t)\leq \overline{R}$ for all $t\in [t^*,t^*+\varepsilon)$. Notice that the case $h=1$ is covered by using the convention $R_0(t)=0$. Similarly, the case $i=N$ is covered by using the convention $R_{N+1}(t)=0$.

We compute, for $t\in (t^*,t^*+\varepsilon)$,
\begin{align*}
    & x_{i}(t)-x_{h-1}(t)=\sum_{j=h-1}^{i-1}\left(x_{j+1}(t)-x_j(t)\right)\\
    & \ = \sum_{j=h-1}^{i-1}\left[x_{j+1}(t^*)-x_j(t^*)+\int_{t^*}^t\left(\dot{x}_{j+1}(s)-\dot{x}_j(s)\right)ds\right]\\
    & \ = \frac{i-h+1}{N\overline{R}}+\int_{t^*}^t\left(\dot{x}_{i}(s)-\dot{x}_{h-1}(s)\right)ds\\
    & \ = \frac{i-h+1}{N\overline{R}}+\int_{t^*}^t\left(\max_{R\in [R_{i+1}(s),R_i(s)]}v(R)-\min_{R\in [R_{h-1}(s),R_h(s)]}v(R)\right)ds\,,
\end{align*}
where we have used
\begin{align*}
    & R_{h-1}(s)\leq \overline{R}<R_h(s)\\
    & R_{i+1}(s)\leq \overline{R}<R_i(s)
\end{align*}
for all $s\in (t^*,t^*+\varepsilon)$. Now, we observe
\begin{align*}
    & \max_{R\in [R_{i+1}(s),R_i(s)]}v(R)\geq \max_{R\in[\max\{R_{i+1}(s),R_{h-1}(s)\},\min\{R_i(s),R_h(s)\}]} v(R)
\end{align*}
and
\begin{align*}
    & \min_{R\in [R_{h-1}(s),R_h(s)]}v(R)\leq \min_{R\in[\max\{R_{i+1}(s),R_{h-1}(s)\},\min\{R_i(s),R_h(s)\}]}v(R)\,,
\end{align*}
which implies the term under the integral sign is nonnegative and we obtain, for all $t\in (t^*,t^*+\varepsilon)$,
\begin{equation}\label{eq:MP_contr_1}
    x_{i}(t)-x_{h-1}(t)\geq \frac{i-h+1}{N\overline{R}}\,.
\end{equation}
The inequality \eqref{eq:MP_contr_1} implies there exists at least one index $j\in\{h,\ldots,i\}$ such that
\[x_{j}(t)-x_{j-1}(t)\geq \frac{1}{N\overline{R}}\,,\]
which implies $R_j(t)\leq \overline{R}$ for all $t\in (t^*,t^*+\varepsilon)$ for some $j\in\{h,\ldots,i\}$, a contradiction. This proves \eqref{eq:MP}. Consequently, two consecutive particles $x_{i-1}(t)$ and $x_{i}(t)$ are distant at least $1/(N\overline{R})$, which implies that the solution $X(t)$ always stays away from the boundary of the closed cone $\mathcal{K}_N$. This implies the solution $X(t)$ can be extended for all times.
\endproof

\section{$BV$ contraction and $L^1$ time equi-continuity}

In this section we still work with the finite dimensional system \eqref{eq:FTL_components} without establishing any connection whatsoever  with the PDE \eqref{eq:intro_CL}. 

Let us now introduce the notation
\begin{equation}\label{eq:density_approx}
    \rho^N(t,x)=\sum_{i=1}^{N}R_i[X(t)]\mathbf{1}_{[x_{i-1}(t),x_i(t))}(x)\,.
\end{equation}
The goal of this section is to prove that the scheme \eqref{eq:FTL_components} is non expansive in the total variation of the discrete density $\rho^N$ and it features an equi-continuity property with respect to time. We observe that
\[
\int_{\R}\rho^N(t,x) dx =1\qquad \hbox{for all $t\geq 0$.}
\]
We recall the following expression for the total variation of $\rho^N(t,\cdot)$ at some fixed time $t\geq 0$:
\[TV[\rho^N(t,\cdot)]=R_1(t)+\sum_{i=1}^{N-1}|R_{i+1}(t)-R_i(t)| + R_N(t)\]
with the usual notation \eqref{eq:short_R}. 
We prove that the above quantity does not grow in time.

\begin{theorem}\label{thm:TV}
    For all $t\geq 0$,
    \begin{equation}\label{eq:TV}
        TV[\rho^N(t,\cdot)]\leq TV[\rho^N(0,\cdot)]\,.
    \end{equation}
\end{theorem}

\proof
We compute
\begin{align*}
    & \frac{d}{dt} TV[\rho^N(t,\cdot)] = \dot{R}_1(t)+\sum_{i=1}^{N-1}\mathrm{sign}(R_{i+1}(t)-R_i(t))(\dot{R}_{i+1}(t)-\dot{R}_i(t)) + \dot{R}_N(t)\\
    & \ = (1-\mathrm{sign}(R_2(t)-R_1(t))\dot{R}_1(t)+\sum_{i=2}^{N-1}\left(\mathrm{sign}(R_i(t)-R_{i-1}(t))-\mathrm{sign}(R_{i+1}(t)-R_i(t))\right)\dot{R}_i(t)\\
    & \ + (1+\mathrm{sign}(R_N(t)-R_{N-1}(t))\dot{R}_N(t)\,.
\end{align*}
We now analyze the term
\begin{align*}
    & I_i(t)=\left(\mathrm{sign}(R_i(t)-R_{i-1}(t))-\mathrm{sign}(R_{i+1}(t)-R_i(t))\right)\dot{R}_i(t)\\
    & \ = -N R_i(t)^2\left(\mathrm{sign}(R_i(t)-R_{i-1}(t))-\mathrm{sign}(R_{i+1}(t)-R_i(t))\right)(\dot{x}_i(t)-\dot{x}_{i-1}(t))\,.
\end{align*}
Let us first consider the case $R_i(t)\geq R_{i+1}(t)$. In this case, \[\dot{x}_i(t)=\max_{R\in[R_{i+1}(t),R_i(t)]}v(R)\,.\]
Since $\mathrm{sign}(R_{i+1}(t)-R_i(t))=-1$, the only case in which $I_i(t)\neq 0$ is the one in which $R_i(t)\geq R_{i-1}(t)$, which implies
\[\dot{x}_{i-1}(t)=\min_{R\in[R_{i-1}(t),R_{i}(t)]}v(R)\,.\]
Hence, in this case we have
\begin{align*}
    & \dot{x}_i(t)-\dot{x}_{i-1}(t)\geq \max_{R\in[\max\{R_{i+1}(t),R_{i-1}(t)\},R_i(t)]}v(R)-\min_{R\in[\max\{R_{i+1}(t),R_{i-1}(t)\},R_i(t)]}v(R)\geq 0\,,
\end{align*}
which implies $I_i(t)\geq 0$. Let us now consider the case $R_i(t)<R_{i+1}(t)$. In this case,
\[\dot{x}_i(t)=\min_{R\in[R_{i}(t),R_{i+1}(t)]}v(R)\,.\]
Since $\mathrm{sign}(R_{i+1}(t)-R_i(t))=1$, the only case in which $I_i(t)\neq 0$ is the one in which $R_i(t)\leq R_{i-1}(t)$, which implies
\[\dot{x}_{i-1}(t)=\max_{R\in[R_{i}(t),R_{i-1}(t)]}v(R)\,.\]
Therefore, in this case we have
\begin{align*}
    & \dot{x}_i(t)-\dot{x}_{i-1}(t)\leq \min_{R\in[R_i(t),\min\{R_{i-1}(t),R_{i+1}(t)\}]}v(R)-\max_{R\in[R_i(t),\min\{R_{i-1}(t),R_{i+1}(t)\}]}v(R)\leq 0\,,
\end{align*}
which yields $I_i(t)\leq 0$. We now analyze the term
\begin{align*}
    & I_1(t)=(1-\mathrm{sign}(R_2(t)-R_1(t))\dot{R}_1(t)\\
    & \ = -NR_1(t)^2 (1-\mathrm{sign}(R_2(t)-R_1(t))(\dot{x}_1(t)-\dot{x}_0(t))\\
    & \ = -NR_1(t)^2 (1-\mathrm{sign}(R_2(t)-R_1(t))\left(\dot{x}_1(t)-\min_{R\in[0,R_1(t)]}v(R)\right)\,.
\end{align*}
If $R_2(t)>R_1(t)$, then $I_1(t)=0$. If $R_2(t)\leq R_1(t)$, then
\[\dot{x}_1(t)=\max_{R\in[R_2(t),R_1(t)]}v(R)\,,\]
which implies $I_1(t)\leq 0$ upon observing that
\[\min_{R\in[0,R_1(t)]}v(R)\leq \min_{R\in[R_2(t),R_1(t)]}v(R)\,.\]
Finally, we analyze the term 
\begin{align*}
    & I_N(t)=(1+\mathrm{sign}(R_N(t)-R_{N-1}(t))\dot{R}_N(t)\\
    & \ = -NR_N(t)^2(1+\mathrm{sign}(R_N(t)-R_{N-1}(t))(\dot{x}_N(t)-\dot{x}_{N-1}(t))\\
    & \ = -NR_N(t)^2(1+\mathrm{sign}(R_N(t)-R_{N-1}(t))(\max_{R\in [0,R_N(t)]}v(R)-\dot{x}_{N-1}(t))\,.
\end{align*}
If $R_{N-1}(t)>R_N(t)$ then $I_N(t)=0$. If, otherwise, $R_N(t)\geq R_{N-1}(t)$, we have
\[\dot{x}_{N-1}=\min_{R\in [R_{N-1}(t),R_N(t)]}v(R)\,.\]
Since
\[\max_{R\in [0,R_N(t)]}v(R)\geq \max_{R\in [R_{N-1}(t),R_N(t)]}v(R)\,,\]
we have $I_N(t)\leq 0$.
\endproof

The next lemma will be useful in the sequel.
\begin{lemma}
    For all $t\geq 0$, there holds
    \begin{align}
       & |\dot{x}_0(t)|+\sum_{i=1}^N|\dot{x}_i(t)-\dot{x}_{i-1}(t)| + |\dot{x}_N(t)|\leq 2v(0)+2[v]_{\mathrm{Lip}([0,\overline{R}])} TV[\rho^N(t,\cdot)]\,,\label{eq:est_vel_BV}\\
       & \sum_{i=1}^N|v(R_i(t))-\dot{x}_{i-1}(t)|\leq [v]_{\mathrm{Lip}([0,\overline{R}])}TV[\rho^N(t,\cdot)]\label{eq:est_vel_BV_2}
    \end{align}
with $\overline{R}$ defined in \eqref{eq:R_bar}
\end{lemma}

\proof
We rewrite \eqref{eq:FTL_components} as
\[\dot{x}_i(t)= \begin{cases}
    v(\hat{R}_{i,i+i}(t)) & i\in\{1,\ldots,N-1\}\\
    v(\hat{R}_1(t)) & i=0\\
    v(\hat{R}_N(t)) & i=N
\end{cases}\]
where $\hat{R}_{i,i+1}(t)$ is the optimal values computed in the single components of \eqref{eq:FTL_components} for $i\in\{1,\ldots,N-1\}$, $\hat{R}_1(t)$ is the minimal value computed in $[0,R_i(t)]$ for $\dot{x}_0(t)$, and $\hat{R}_N(t)$ is the maximal value computed in $[0,R_N(t)]$ for $\dot{x}_N(t)$. All the above optimal values are always achieved since $v$ is continuous. They satisfy
\begin{align*}
    & \max\{|\hat{R}_{i,i+1}(t)-R_i(t)|+|\hat{R}_{i,i+1}(t)-R_{i+1}(t)|\}\leq |R_{i+1}(t)-R_i(t)|\qquad \hbox{for $i=1,\ldots,N-1$}\,,\\
    & \hat{R}_1(t)\leq R_1(t)\,,\\
    & \hat{R}_N(t)\leq R_N(t)\,.
\end{align*}
Hence,
    \begin{align*}
        & |\dot{x}_0(t)|+\sum_{i=1}^N|\dot{x}_i(t)-\dot{x}_{i-1}(t)| + |\dot{x}_N(t)|\\
    & \ \leq |v(\hat{R}_{1})|+\sum_{i=1}^N|v(\hat{R}_{i,i+1}(t))-v(\hat{R}_{i-1,i}(t))|+ |v(\hat{R}_N(t))|\\
    & \ \leq 2|v(0)|+[v]_{\mathrm{Lip}([0,\overline{R}])}\hat{R}_1(t)+ [v]_{\mathrm{Lip}([0,\overline{R}])}\sum_{i=1}^N|\hat{R}_{i,i+1}(t)-\hat{R}_{i-1,t}(t)| + [v]_{\mathrm{Lip}([0,\overline{R}])}\hat{R}_N(t)\\
    & \ \leq 2|v(0)|+[v]_{\mathrm{Lip}([0,\overline{R}])}\left(R_1(t)+R_N(t)\right)\\
    &  \ + [v]_{\mathrm{Lip}([0,\overline{R}])}\sum_{i=1}^N\left(|\hat{R}_{i,i+1}(t)-R_i(t)|+|R_i(t)-\hat{R}_{i-1,i}(t)|\right)\\
    & \ \leq 2|v(0)|+[v]_{\mathrm{Lip}([0,\overline{R}])}\left(R_1(t)+R_N(t)\right)\\
    &  \ + [v]_{\mathrm{Lip}([0,\overline{R}])}\sum_{i=1}^N\left(|R_{i+1}(t)-R_i(t)|+|R_i(t)-R_{i-1}(t)|\right)\,,
    \end{align*}
    which implies \eqref{eq:est_vel_BV}. We then observe
    \begin{align*}
        & \sum_{i=1}^N|v(R_i(t))-\dot{x}_{i-1}(t)|= \sum_{i=2}^N|v(R_i(t))-v(\hat{R}_{i-1,i}(t))|+|v(R_1(t))-v(\hat{R}_1(t))|\\
        & \ \leq [v]_{\mathrm{Lip}([0,\overline{R}])}\sum_{i=2}^N |R_i(t)-\hat{R}_{i-1,i}(t)|+[v]_{\mathrm{Lip}([0,\overline{R}])}|R_1(t)-\hat{R}_1(t)|\\
        & \ \leq [v]_{\mathrm{Lip}([0,\overline{R}])}\sum_{i=2}^N |R_i(t)-R_{i-1}(t)|+[v]_{\mathrm{Lip}([0,\overline{R}])}R_1(t)\leq [v]_{\mathrm{Lip}([0,\overline{R}])}TV[\rho^N(t,\cdot)]
    \end{align*}
    which proves \eqref{eq:est_vel_BV_2}.
\endproof

We now focus on the time-variation of the scheme. We introduce the function $X^N:[0,+\infty)\times [0,1]\rightarrow \R$ defined for $z\in [0,1)$ by
\[
X^N(t,z)=\sum_{i=1}^N \left(x_{i-1}(t)+R_i(t)^{-1}\left(z-\frac{i-1}{N}\right)\right)\mathbf{1}_{[(i-1)/N,i/N))}(z)
\]
extended by continuity on $z=1$. For $0\leq s\leq t$, we recall (see \cite{villani}) that the $1$-Wasserstein distance between $\rho^N(t,\cdot)$ and $\rho^N(s,\cdot)$ is provided by 
\[d_1(\rho^N(t,\cdot),\rho^N(s,\cdot))=\|X^N(t,\cdot)-X^N(s,\cdot)\|_{L^1{(0,1)}}\,.\]

\begin{proposition}\label{prop:time_cont}
    There exists a constant $C\geq 0$ depending neither on $N$ nor on time such that
    \begin{equation}\label{eq:equi_cont}
        \|\rho^N(t,\cdot)-\rho^N(s,\cdot)\|_{L^1(\R)}\leq C TV(\rho^N(0,\cdot))^{1/2}|t-s|^{1/2}\,.
    \end{equation}
\end{proposition}

\proof 
We compute
\begin{align*}
    & d_1(\rho^N(t,\cdot),\rho^N(s,\cdot))= \left|\sum_{i=1}^N\int_{(i-1)/N}^{i/N}\left(x_{i-1}(t)-x_{i-1}(s))+\left(R_{i}(t)^{-1}-R_i(s)^{-1}\right)\left(z-\frac{i-1}{N}\right)\right)dz\right|\\
    & \ \leq \sum_{i=1}^N\left[\frac{1}{N}\int_s^t |\dot{x}_{i-1}(\tau)|d\tau+\frac{1}{N}\int_s^t|\dot{x}_i(\tau)-\dot{x}_{i-1}(\tau)|d\tau \right]\,.
\end{align*}
Since $v$ is continuous and always evaluated on the interval $[0,\overline{R}]$ in \eqref{eq:FTL_components}, easily obtain 
\[d_1(\rho^N(t,\cdot),\rho^N(s,\cdot))\leq C|t-s|\]
for a suitable $C\geq 0$ that depends neither on time nor on $N$. We now recall the following inequality from \cite[Lemma A.1]{DFIR}, which holds for two probability densities $\rho_1, \rho_2$
\begin{equation}\label{eq:interp}
    \|\rho_1-\rho_2\|_{L^1(\R)}\leq 2(TV[\rho_1]+TV[\rho_2])^{1/2}d_1(\rho_1,\rho_2)^{1/2}
\end{equation}
which implies the assertion with $\rho_1=\rho^N(t,\cdot)$ and $\rho_2=\rho^N(s,\cdot)$.
\endproof

\section{Approximation of entropy solutions}\label{sec:entropy}

In this section we finally address the resolution of the Cauchy problem for \eqref{eq:intro_CL}, in the entropy sense \cite{kruzkov}, via the Follow-the-Leader scheme \ref{eq:FTL_components}. Let us first define our set of initial data.
For a constant $\alpha>0$, we define the set
\[\mathcal{M}_\alpha(\R)=\left\{\rho\in L^1(\R)\,:\,\, \int_\R|x|^\alpha \rho(x) dx <+\infty\right\}\,.\]
We then define the set
\[
\displaystyle{\mathcal{X}(\R)=\left\{\rho\in L^1(\R)\cap L^\infty(\R)\cap BV(\R)\cap \mathcal{M}_\alpha(\R)\ \hbox{for some $\alpha>0$}\,:\,\, \hbox{$\rho_0\geq 0$ and $\displaystyle{\int_\R\rho_0(y)dy = 1}$}\right\}}
\]
We consider a  \emph{sampling map}, that is an association of a probability density in $\mathcal{X}(\R)$ 
to an ordered $(N+1)$-vector of non-overlapping particle positions,
\begin{equation}\label{eq:initial_sampling_2}
  \sam^N:\mathcal{X}(\R)\rightarrow \mathring{\mathcal{K}}_N\,,\quad \sam^N[\rho]=\left(\sam^N_0[\rho],\ldots,\sam^N_N[\rho]\right)\,.
\end{equation}
We then define 
\[\mathbf{R}^N[\rho]=\mathcal{R}[\sam^N[\rho]]\]
where $\mathcal{R}$ is the map defined in \eqref{eq:R1}-\eqref{eq:R2}. 
We assume the following approximation properties. First of all, we require the weak measure approximation property
\begin{equation}\label{eq:initial_sampling}
  \int_\R \varphi(x)\mathbf{R}^N[\rho](x)dx\rightarrow \int_\R \varphi(x)\rho(x) dx\qquad \hbox{as $N\rightarrow+\infty$, for all $\varphi\in C^0(\R)$}\,.
\end{equation}
Moreover, we require
\begin{equation}\label{eq:approx_contraction_1}
    \|\mathbf{R}^N[\rho]\|_{L^\infty}\leq \|\rho\|_{L^\infty}
\end{equation}
and
\begin{equation}\label{eq:approx_contraction_2}
    TV[\mathbf{R}^N[\rho]]\leq TV[\rho]\,.
\end{equation}
We observe that the assumption $\rho\in \mathcal{M}_\alpha(\R)$ implies, by a simple lower semi-continuity of the $\alpha$-moment with respect to the weak$^*$ measure topology, 
\begin{equation}\label{eq:sampling_moment}
\limsup_{N\rightarrow+\infty}\int_\R|x|^\alpha \mathbf{R}^N[\rho](x) dx\leq \int_\R|x|^\alpha \rho(x) dx\,.
\end{equation}
There are a variety of possible sampling maps with the above properties.
A frequently used one, which we denote by $\sam^N_*$,
applies to compactly supported measures $\rho\in \mathcal{X}(\R)$ and is defined as follows.
First set
\begin{equation}\label{eq:sampling_support_0}
    \sam^N_{*,0}[\rho]=\inf{\mathrm{supp}(\rho)},
\end{equation}
and then recursively for $i=1,\ldots,N$:
\begin{equation}\label{eq:sampling_support_i}
   \sam^N_{*,i}[\rho]=\inf\left\{x\geq \sam^N_{*,i-1}[\rho]\,:\,\, \int_{\sam^N_{*,i-1}[\rho]}^x d\rho\geq \frac{1}{N}\right\}\,.
\end{equation}
We observe that $\sam^N_*[\rho]\in \mathring{\mathcal{K}}_N$ and that $\sam^N_{*,N}[\rho]=\sup \mathrm{supp}(\rho)$.
The map $\sam^N_*$ is usually called  \emph{support preserving sampling}. We refer to \cite{DR} for the proof that $\sam^N_\ast$ satisfies the above properties \eqref{eq:initial_sampling}-\eqref{eq:approx_contraction_1}-\eqref{eq:approx_contraction_2}. 
From now on we assume that a sampling map $\sam^N$ satisfying \eqref{eq:initial_sampling}-\eqref{eq:approx_contraction_1}-\eqref{eq:approx_contraction_2} is fixed. 

For a given $\rho_0\in \mathcal{X}(\R)$, we set $X_0=(x_{0,1},\ldots,x_{0,N})=\sam^N[\rho_0]$ and 
$\rho_0^N=\mathbf{R}^N[\rho_0]$ and consider the solution $X(t)=(x_0(t),\ldots,x_N(t))$ to the Cauchy problem \eqref{eq:FTL_cauchy} with initial datum $X_0$. We then set reconstruct the time-depending approximated density $\rho^N$ as in \eqref{eq:density_approx}.

The result in Theorem \ref{thm:MP} and property \eqref{eq:approx_contraction_1} imply
\begin{equation}\label{eq:uniform_est_MP}
    \|\rho^N(t,\cdot)\|_{L^\infty(\R)}\leq \|\rho_0\|_{L^\infty(\R)}\qquad \hbox{for all $N\in \mathbb{N}$ and for all $t\geq 0$}\,.
\end{equation}
Moreover, Theorem \ref{thm:TV} and the property \eqref{eq:approx_contraction_2} imply
\begin{equation}\label{eq:uniform_est_BV}
    TV[\rho^N(t,\cdot)]\leq TV[\rho_0]\qquad \hbox{for all $N\in \mathbb{N}$ and for all $t\geq 0$}\,.
\end{equation}
Finally, Proposition \ref{prop:time_cont} and the estimate \eqref{eq:uniform_est_BV} imply
\begin{equation}\label{eq:equi}
    \|\rho^N(t,\cdot)-\rho^N(s,\cdot)\|_{L^1(\R)}\leq C TV(\rho_0)^{1/2}|t-s|^{1/2}\,.
\end{equation}

\begin{theorem}[Strong compactness]\label{thm:compactness}
   Let $\rho_0\in \mathcal{X}(\R)$. Let $T\geq 0$ and let $K\subset \R$ be a compact set. Then, the sequence $\rho^N$ is strongly compact in $C([0,T];\,L^1(\R))$ equipped with the norm $\sup_{t\in [0,T]}\|\rho(t,\cdot)\|_{L^1(K)}$. Every limit function $\rho$ is a.e. nonnegative and satisfies $\rho(t,\cdot)\in \mathcal{X}(\R)$ for all $t\geq 0$.
\end{theorem}

\proof
Since $\|\rho^N\|_{L^1(\R)}=1$ for all $N\in \N$ and due to the equi-continuity property \eqref{eq:equi}, the curve $[0,T]\ni t\mapsto \rho^N(t,\cdot)$ is uniformly bounded and equi-continuous from $[0,T]$ with values in  $L^1(\R)$, and hence in $L^1(K)$. Moreover, the estimate \eqref{eq:uniform_est_BV} implies that the image of the curve $[0,T]\ni t\mapsto \rho^N(t,\cdot)$ lies in a compact subset of $L^1(K)$. Hence, Arzelà-Ascoli's Theorem (see e.g. \cite[Proposition 3.3.1]{AGS}) implies the compactness assertion. The proof of the nonnegativity follows from the strong $L^1([0,T]\times K)$ compactness. The total variation bound follows from the lower semi-continuity of the total variation with respect to the $L^1$ topology. The $L^\infty$ bound follows from weak$^*$ $L^\infty$ compactness and standard $L^\infty$-weak$^*$ lower semi-continuity of the $L^\infty$ norm. Finally, for $\alpha>0$ such that $\rho_0\in\mathcal{M}(\R)$, we estimate for every $t\geq 0$
\begin{align*}
    & \int_\R |x|^\alpha \rho^N(t,x) dx = \sum_{i=1}^N R_i(t) \int_{x_{i-1}(t)}^{x_i(t)}|x|^\alpha dx\leq \frac{1}{N}\sum_{i=1}^N |x_i(t)|^\alpha\,.
\end{align*}
Since every particle $x_i(t)$ moves by a speed which is controlled in absolute value by the maximum of $v$ on the interval $[0,R]$, the above $\alpha$-moment is uniformly bounded with respect to $N$ on a fixed time interval $[0,T]$ provided it is for $t=0$, which is the case due to the assumption $\rho_0\in\mathcal{M}(\R)$. Hence, $\rho^N$ is tight in the space of probability measures on $\R$ due to Prokhorov's Theorem \cite{AGS}, which implies the limit $\rho$ is a probability density (and in particular, it has unit mass) for all times.
\endproof

Our next goal is to identify the up-to-a-subsequence limit $\rho$ of the sequence $\rho^N$ provided by Theorem \ref{thm:compactness}. We recall the following definition.

\begin{definition}\label{def:entropy}
    Let $f(\rho)=\rho v(\rho)$ and let $\rho\in \mathcal{X}(\R)$. We say that a measurable function $\rho:[0,+\infty)\times \R\rightarrow \R$ is a weak solution to the Cauchy problem
    \begin{equation}\label{eq:cauchy_CL}
    \begin{cases}
    \rho_t+f(\rho)_x = 0 & \\
    \rho(0,x)=\rho_0(x)
    \end{cases}
    \end{equation}
    is, for all $\varphi\in C^1_c([0,+\infty)\times \R)$ there holds
    \begin{align}
        & \int_0^{+\infty}\int_\R \left[\rho(t,x)\varphi_t(t,x)+f(\rho(t,x))\varphi_x(t,x)\right]dx dt+\int_\R\rho_0(x)\varphi(0,x)dx = 0\,.\label{eq:weak}
    \end{align}
    The measurable function $\rho$ is called an entropy solution to the Cauchy problem \eqref{eq:cauchy_CL} if, for all $\varphi\in C^1_c([0,+\infty)\times \R)$ with $\varphi\geq 0$ and for all $k\in \R$,  there holds
 \begin{align}
        & \int_0^{+\infty}\int_\R \left[\left|\rho(t,x)-k\right|\varphi_t(t,x)+\mathrm{sign}(\rho(t,x)-k)\left(f(\rho(t,x)-f(k)\right)\varphi_x(t,x)\right]dx dt\nonumber\\
        & \quad +\int_\R\left|\rho_0(x)-k\right|\varphi(0,x)dx \geq 0\,.\label{eq:entropy}
    \end{align}
\end{definition}
It is well known from \cite{kruzkov} that for a given initial condition $\rho_0\in L^1(\R)\cap L^\infty(\R)$ there exists a unique entropy solution as in Definition \ref{def:entropy}.

\medskip
We now state the main result of this paper.
\begin{theorem}\label{thm:main}
    Let $\rho_0\in \mathcal{X}(\R)$. Then, the approximated density $\rho^N$ constructed in \eqref{eq:density_approx} from the Cauchy problem \eqref{eq:FTL_cauchy} with initial datum $X_0=\sam^N[\rho_0]$ converges strongly in $C([0,T];\,L^1(\R))$ to the unique entropy solution in the sense of Definition \ref{def:entropy}.
\end{theorem}

\proof
Similarly to \cite{DR} and \cite{DFR}, for a given $\varphi$ as in the assumptions we compute
\begin{align}
    & \int_0^{+\infty}\int_\R \left[\left|\rho^N(t,x)-k\right|\varphi_t(t,x)+\mathrm{sign}(\rho^N(t,x)-k)\left(f(\rho^N(t,x)-f(k)\right)\varphi_x(t,x)\right]dx dt\nonumber\\
         & \quad +\int_\R\left|\rho^N(0,x)-k\right|\varphi(0,x)dx= I_1 + I_2 + I_3\label{eq:entropy_approx}
\end{align}
with
\begin{align*}
        & I_1= \sum_{i=1}^N\int_0^{+\infty}\mathrm{sign}(R_i(t)-k)R_i(t)(\dot{x}_i(t)-\dot{x}_{i-1}(t))\left(\fint_{x_{i-1}(t)}^{x_i(t)}\varphi(t,x) dx-\varphi(t,x_i(t))\right)dt\\
        & I_2 =\sum_{i=1}^N\int_0^{+\infty}(\mathrm{sign}(R_i(t)-k))\left[R_i(t)(\dot{x}_i(t)-\dot{x}_{i-1}(t))\varphi(t,x_i(t))\right.\\
        & \ \quad \left.\vphantom{\int}-(R_i(t)-k)(\varphi(t,x_i(t))\dot{x}_i(t)-\varphi(t,x_{i-1}(t))\dot{x}_{i-1}(t))\right.\\
        & \ \quad \left.\vphantom{\int}+(f(R_i(t))-f(k))(\varphi(t,x_i(t))-\varphi(t,x_{i-1}(t)))\right]dt\\
        & I_3=\int_0^{+\infty}\int_{-\infty}^{x_0(t)}\left(|k|\varphi_t(t,x) +\mathrm{sign}(k)f(k)\varphi_x(t,x)\right)dx dt\\
        & \ \quad + \int_0^{+\infty}\int_{X_N(t)}^{+\infty}\left(|k|\varphi_t(t,x) +\mathrm{sign}(k)f(k)\varphi_x(t,x)\right)dx dt\\
        & \ + \int_{-\infty}^{x_0(t)}|k|\varphi(0,x) dx + \int_{x_N(t)}^{+\infty}|k|\varphi(0,x) dx\,.
\end{align*}
Here, in principle, $\rho^N$ is the subsequence arising from Theorem \ref{thm:compactness} converging strongly in $L^1_{\mathrm{loc}}([0,+\infty)\times \R)$. 
Due to the assumptions on $\varphi$ and due to the result in Theorem \ref{thm:compactness}, all terms in \eqref{eq:entropy_approx} can be passed to the limit to obtain, as $N\rightarrow$, the expression
\begin{align*}
     & \int_0^{+\infty}\int_\R \left[\left|\rho(t,x)-k\right|\varphi_t(t,x)+\mathrm{sign}(\rho(t,x)-k)\left(f(\rho(t,x)-f(k)\right)\varphi_x(t,x)\right]dx dt\nonumber\\
        & \quad +\int_\R\left|\rho_0(x)-k\right|\varphi(0,x)dx \geq 0\,.
\end{align*}
To conclude the assertion, we only have to prove that $I_1+I_2+I_3$ converge in the $N\rightarrow+\infty$ limit to a nonnegative number. Since the entropy solution according to Definition \ref{def:entropy} is unique, we would then deduce that the whole sequence $\rho^N$ converges to the same limit.

We start by estimating the term $I_1$ as follows. Let $T\geq 0$ be such that $\mathrm{supp}(\varphi)\subset [0,T]\times \R$. We have
\begin{align*}
    & |I_1|\leq [\varphi]_{\mathrm{Lip}}\sum_{i=1}^N\int_0^{T}R_i(t)(x_i(t)-x_{i-1}(t))|\dot{x}_i(t)-\dot{x}_{i-1}(t)|dt\\
    & \ = \frac{T[\varphi]_{\mathrm{Lip}}}{N}\sup_{0\leq t\leq T}\sum_{i=1}^N |\dot{x}_i(t)-\dot{x}_{i-1}(t)|\,.
\end{align*}
Hence, the estimate \eqref{eq:est_vel_BV} implies $I_1\rightarrow 0$ as $N\rightarrow +\infty$.

We now compute the term $I_2$. After one cancelation, the term becomes
\begin{align*}
    & I_2=\sum_{i=1}^N\int_0^{+\infty}\mathrm{sign}(R_i(t)-k)\left[-R_i(t)\dot{x}_{i-1}\varphi(t,x_i(t))+R_i(t)\dot{x}_{i-1}(t)\varphi(t,x_{i-1}(t))\right.\\
    & \ \left.+k\dot{x}_i(t)\varphi(t,x_i(t))-k\dot{x}_{i-1}(t)\varphi(t,x_{i-1}(t))+R_i(t)v(R_i(t))\varphi(t,x_i(t))\right.\\
    & \ \left.-R_i(t)v(R_i(t))\varphi(t,x_{i-1}(t))-kv(k)\varphi(t,x_i(t))+kv(k)\varphi(t,x_{i-1}(t))
    \right]\\
    & \ = J_1+J_2
\end{align*}
with
\begin{align*}
& J_1 = \sum_{i=1}^N\int_0^{+\infty}\mathrm{sign}(R_i(t)-k)(\varphi(t,x_i(t))-\varphi(t,x_{i-1}(t)))R_i(t)(v(R_i(t))-\dot{x}_{i-1}(t))dt\\
    & J_2 = k\sum_{i=1}^N\int_0^{+\infty}\mathrm{sign}(R_i(t)-k)\left(\varphi(t,x_i(t))(\dot{x}_i(t)-v(k)) -\varphi(t,x_{i-1}(t))(\dot{x}_{i-1}-v(k))\right)dt\,.
\end{align*}
As for $J_1$, we have the estimate
\begin{align*}
    & |J_1|\leq [\varphi]_{\mathrm{Lip}}\sum_{i=1}^N\int_0^{+\infty}(x_i(t)-x_{i-1}(t))R_i(t)|v(R_i(t))-\dot{x}_{i-1}(t)|dt\\
    & \ =\frac{T [\varphi]_{\mathrm{Lip}}}{N}\sup_{0\leq t\leq T}\sum_{i=1}^N |v(R_i(t))-\dot{x}_{i-1}(t)|\,
\end{align*}
and \eqref{eq:est_vel_BV_2} implies $J_1\rightarrow 0$ as $N\rightarrow +\infty$. The term $J_2$ can be rearranged via summation by parts as follows:
\begin{align*}
    & J_2 = k\sum_{i=1}^{N-1}\int_0^{+\infty}\varphi(t,x_i(t))\left(\mathrm{sign}(R_i(t)-k)-\mathrm{sign}(R_{i+1}(t)-k)\right)(\dot{x}_i(t)-v(k))dt\\
    & \ + k\int_0^{+\infty}\left(\varphi(t,x_N(t))\mathrm{sign}(R_N(t)-k)(\dot{x}_N(t)-v(k)) -\varphi(t,x_0(t))(\mathrm{sign}(R_1(t)-k)(\dot{x}_0(t)-v(k))\right)dt\,.
\end{align*}
As for the main generic term
\[J_2^i=k
\left(\mathrm{sign}(R_i(t)-k)-\mathrm{sign}(R_{i+1}(t)-k)\right)(\dot{x}_i(t)-v(k))\,,\qquad i=1,\ldots,N-1\,,
\]
we observe that said term equals zero unless $k$ is a convex combination of $R_i(t)$ and $R_{i+1}(t)$. In particular, we may assume without restriction that $k\geq 0$. Assume first $R_{i+1}(t)\leq R_i(t)$, which implies $k\in [R_{i+1}(t),R_i(t)]$. Then $\dot{x}_i(t)=\max_{R\in [R_{i+1}(t),R_i(t)]}v(R)\geq v(k)$. Since in this case $\left(\mathrm{sign}(R_i(t)-k)-\mathrm{sign}(R_{i+1}(t)-k)\right)=2$, we then have $J_2^i\geq 0$. Assume now $R_i(t)\leq R_{i+1}(t)$ and $k\in [R_i(t),R_{i+1}(t)]$. Then, $\dot{x}_i(t)=\min_{R\in [R_{i}(t),R_{i+1}(t)]}v(R)\leq v(k)$, and since in this case $\left(\mathrm{sign}(R_i(t)-k)-\mathrm{sign}(R_{i+1}(t)-k)\right)=-2$, we still have $J_2^i\geq 0$. Therefore, $I_1+I_2+I_3$ amounts to terms that converge to zero for large $N$, plus nonnegative terms, plus the remaining terms 
\begin{align*}
    & I_3 +k\int_0^{+\infty}\left(\varphi(t,x_N(t))\mathrm{sign}(R_N(t)-k)(\dot{x}_N(t)-v(k)) -\varphi(t,x_0(t))(\mathrm{sign}(R_1(t)-k)(\dot{x}_0(t)-v(k))\right)dt\\
    & \ = -|k|\int_0^{+\infty}\dot{x}_0(t)\varphi(t,x_0(t))dt+|k|\int_0^{+\infty}\dot{x}_N(t)\varphi(t,x_N(t))dt \\
    & \ +|k|v(k)\int_{0}^{+\infty}\varphi(t,x_0(t))dt-|k|v(k)\int_{0}^{+\infty}\varphi(t,x_N(t))dt\\
    & \ +k\int_0^{+\infty}\left(\varphi(t,x_N(t))\mathrm{sign}(R_N(t)-k)(\dot{x}_N(t)-v(k)) -\varphi(t,x_0(t))(\mathrm{sign}(R_1(t)-k)(\dot{x}_0(t)-v(k))\right)dt\\
    & \ = \int_0^{+\infty}\varphi(t,x_0(t))\left(-|k|\dot{x}_0(t)+|k|v(k)-k\mathrm{sign}(R_1(t)-k)(\dot{x}_0(t)-v(k))\right)dt\\
    & \ + \int_0^{+\infty}\varphi(t,x_N(t))\left(|k|\dot{x}_N(t)-|k|v(k)+k\mathrm{sign}(R_N(t)-k)(\dot{x}_N(t)-v(k))\right)dt\\
    & \ = k\int_0^{+\infty}\varphi(t,x_0(t))(\dot{x}_0(t)-v(k))(-\mathrm{sign}(k)-\mathrm{sign}(R_1(t)-k))dt\\
    & \ + k\int_0^{+\infty}\varphi(t,x_N(t))(\dot{x}_N(t)-v(k))(\mathrm{sign}(k)+\mathrm{sign}(R_N(t)-k))dt\,.
\end{align*}
The first term above is possibly nonzero only if $k\in [0,R_1(t)]$, in which case $(-\mathrm{sign}(k)-\mathrm{sign}(R_1(t)-k))=-2$. Since $\dot{x}_0(t)=\min_{R\in [0,R_1(t)]}v(R)\leq v(k)$, the first term above is nonnegative. Similarly, the second term is possibly nonzero only if $k\in [0,R_N(t)]$, which implies $(\mathrm{sign}(k)+\mathrm{sign}(R_N(t)-k))=2$. Since $\dot{x}_N(t)=\max_{R\in [0,R_N(t)]}v(R)\geq v(k)$, the second term is also nonnegative.
\endproof

\section*{Appendix}\label{sec:app}

\begin{proof}[Proof of Lemma \ref{lem:max_min}]
On the domain $D=\left\{(a,b)\in \R^2\,:\,\, a\leq b\right\}$, the map $F$ is trivially Lipschitz continuous with respect to $a$ because the map $[0,b]\ni a \mapsto F(a,b)$ coincides in this case with the non-decreasing hull of $v$ on $[0,b]$. Similarly, the map $[a,+\infty)\ni b \mapsto F(a,b)$ being the non-decreasing hull of $v$ on $[a,+\infty)$ implies the Lipschitz continuity with respect to $b$. A similar argument shows the (joint) Lipschitz continuity of $F$ on $\R^2\setminus D$. The only case left to check the Lipschitz continuity with respect to $a$ is the one in which $(a_1,b)\in D$ and $(a_2,b)\not\in D$. In that case, for some $\hat{a}_1\in [a_1,b]$ and for some $\hat{a}_2\in [b,a_2]$, we have
\begin{align*}
    & |F(a_2,b)-F(a_1,b)|=|v(\hat{a}_2)-v(\hat{a}_1)|\leq [v]_{\mathrm{Lip}[a_1,a_2]}|\hat{a}_1-\hat{a}_2|\leq [v]_{\mathrm{Lip}[a_1,a_2]}|a_1-a_2|\,.
\end{align*}
A similar estimate applies to the case $b_1\leq a\leq b_2$, which completes the proof.
\end{proof}

\section*{Acknowledgments}
The author is partially supported by the Italian \enquote{National Centre for HPC, Big Data and Quantum Computing} - Spoke 5 \enquote{Environment and Natural Disasters}.  
The author is also partially supported by the InterMaths Network, \url{www.intermaths.eu}.
The author acknowledges the use of the AI assistant Claude (Anthropic) during the early stage of this project, in identifying the connection between the follow-the-leader scheme and Godunov-type numerical Hamiltonians for Hamilton–Jacobi equations, and in locating the relevant references \cite{bardi_osher,crandall_lions,leveque}. All mathematical content, including proofs and their verification, was independently developed and is the sole responsibility of the author.


\end{document}